\documentclass{amsproc}

\usepackage{amssymb}

\usepackage[cmtip,all]{xy}

\usepackage{tensor}

\newtheorem{theorem}{Theorem}[section]
\newtheorem{lemma}[theorem]{Lemma}
\newtheorem{corollary}{Corollary}[section]

\theoremstyle{definition}

\newtheorem{example}[theorem]{Example}

\theoremstyle{remark}
\newtheorem{remark}[theorem]{Remark}

\numberwithin{equation}{section}

\begin{document}

\title{Modules over semisymmetric quasigroups}


\author{Alex W. Nowak}
\address{Department of Mathematics, Iowa State University, Ames, Iowa, 50011}
\curraddr{}
\email{anowak@iastate.edu}
\thanks{}



\date{November 15, 2017}

\begin{abstract}
The class of semisymmetric quasigroups is determined by the identity $(yx)y=x.$  We prove that the universal multiplication group of a semisymmetric quasigroup $Q$ is free over its underlying set and then specify the point-stabilizers of an action of this free group on $Q$.  A theorem of Smith indicates that Beck modules over semisymmetric quasigroups are equivalent to modules over a quotient of the integral group algebra of this stabilizer.  Implementing our description of the quotient ring, we provide some examples of semisymmetric quasigroup extensions.  Along the way, we provide an exposition of the quasigroup module theory in more general settings.  
\end{abstract}

\maketitle



\section{Introduction}
\noindent A \emph{quasigroup} $(Q, \cdot, /, \backslash)$ is a set $Q$ equipped with three binary operations: $\cdot$ denoting multiplication, $/$ right division, and $\backslash$ left division; these operations satisfy the identities
\begin{align*}
 \text{(IL)}\phantom{A}y\backslash(y\cdot x)=x; && \text{(IR)}\phantom{A}x= (x\cdot y)/y;\\
 \text{(SL)}\phantom{A}y\cdot(y\backslash x)=x; && \text{(SR)}\phantom{A} x= (x/y)\cdot y.
\end{align*}
Whenever possible, we convey multiplication of quasigroup elements by concatenation.\\
\indent As a class of algebras defined by identities, quasigroups constitute a variety (in the universal-algebraic sense), denoted by $\mathbf{Q}$.  Moreover, we have a category of quasigroups (also denoted $\mathbf{Q}$), the morphisms of which are quasigroup homomorphisms respecting the three operations.   \\
\indent If $(G, \cdot)$ is a group, then defining $x/y:=x\cdot y^{-1}$ and $x\backslash y:=x^{-1}\cdot y$ furnishes a quasigroup $(G, \cdot, /, \backslash)$.  That is, all groups are quasigroups, and because of this fact, constructions in the representation theory of quasigroups often take their cue from familiar counterparts in group theory.  Transporting group modules into a nonassociative setting seems particularly cumbersome, as the specification of a $G$-module $M$ via the structure map $\rho:G\to \text{Aut}(M)$ relies on the associativity of automorphism composition.  However, in \cite{infgrp}, Smith develops an approach to quasigroup module theory that avoids this complication.\\
\indent We have two equivalent interpretations of quasigroup modules (appearing below as Theorems \ref{FTQR} and \ref{RFTQR}; cf. \cite{IQTR}, Section 10.3), and while the focus of this paper will be on the more concrete of the two, we now review the category-theoretical definition, as it best illustrates the connection to group representation theory.  Indeed, notice that a right $G$-module $M$ gives rise to the split extension $E=G\tensor*[]{\ltimes}{}M$ built upon the set $G\times M$ with product $$(g_1, m_1)(g_2, m_2)=(g_1g_2, m_1g_2+m_2),$$ accompanied by projection $\pi: (g, m)\mapsto g$ and section $\eta:g\mapsto (g, 0).$  The former map $\pi:E\to G$ may be construed as an object of the slice category $\mathbf{Gp}/G$ of group homomorphisms over $G$.  Indeed, $\eta:G\to E$ constitutes a zero map, $$+_G:E\times_G E\to E; ((g, m_1), (g, m_2))\mapsto (g, m_1+m_2)$$ from the pullback in $\mathbf{Gp}$ over $G$ into $E$ plays the role of addition, and the map $-_G:E\to E; (g, m)\mapsto(g, -m)$ provides negation.  Conversely, if one starts with an abelian group object $\pi:E\to G$ in $\mathbf{Gp}/G$ with addition $+:E\times_G E\to E$, negation $-:E\to E$, and zero map $\eta:G\to E,$ then $M=\text{Ker } \pi$ is an abelian group because $+$ satisfies the commutativity axiom for abelian group objects.  Moreover, $G$ acts on $M$ via conjugation: $$\rho_g:M\to M; m\mapsto  (g^\eta)^-mg^\eta$$ (concatenation represents multiplication in the group $E$).  \\
\indent The equivalence of $G$-modules and abelian group objects in $\mathbf{Gp}/G$ is an example of a more general construction due to Beck \cite{Beck}.  In fact, if $\mathbf{C}$ is a category with pullbacks and $Q$ is an object of $\mathbf{C}$, a \emph{Beck module} over $Q$ is defined to be an abelian group object in the slice category $\mathbf{C}/Q$ of $\mathbf{C}$-morphisms over $Q$.  The category of quasigroups $\mathbf{Q}$, as a category of universal algebras defined by identities, is bicomplete, and, thus, has pullbacks.  In fact, any subvariety $\mathbf{V}$ of $\mathbf{Q}$ specified by identities (in addition to the quasigroup axioms) has, when construed as a category, pullbacks.  It is with this fact at our disposal that, given a quasigroup $Q$ in a variety $\mathbf{V}$, we define a $Q$\emph{-module in} $\mathbf{V}$ to be an abelian group object in the slice category $\mathbf{V}/Q$ of $\mathbf{V}$-morphisms over $Q$.  \\
\indent The purpose of this paper is to describe modules over \emph{semisymmetric quasigroups}, those which satisfy the identity $(yx)y=x$ and thus constitute a category that we denote by $\mathbf{P}$.  Abelian group objects in $\mathbf{P}/Q$ are equivalent to modules over the quotient of an integral group ring.  An explicit description of this ring is our main result, appearing as Theorem \ref{mainresult}.  Before arriving at this theorem, we devote Section 2 to a general discussion of semisymmetric quasigroups; we provide some small examples and explain the significance of semisymmetry in two applications of nonassociative algebra.  The purpose of Section 3 is to introduce the reader to \emph{universal multiplication groups}, variety-dependent objects that associate, in a manner that is functorial, a quasigroup $Q$ to a group of permutations on a free extension of $Q$ over an indeterminate.  Theorem \ref{unimult} is our first result.  We prove that the universal multiplication group of a semisymmetric quasigroup is free over the its underlying set.  Theorem \ref{ustab} then describes a basis for the stabilizer of points in $Q$ under an action of the universal multiplication group; the representation ring for semisymmetric quasigroups is a quotient of the integral group algebra of this \emph{universal stabilizer}.  Establishing Theorem \ref{mainresult} then becomes a matter of determining the ideal by which we must quotient.  Towards this end, we employ techniques of combinatorial differentiation of quasigroup words established by Smith in \cite{IQTR}. We conclude with some examples that illustrate the utility of Theorem \ref{mainresult} in computing extensions of semisymmetric quasigroups.

\section{Semisymmetric quasigroups}
We say that a quasigroup $(Q, \cdot, /, \backslash)$ is \emph{semisymmetric} if it satisfies the following equivalent identities:
\begin{equation}
\label{semisymmid1}
(yx)y=x;
\end{equation}
\begin{equation}
\label{semmisymmid2}
y(xy)=x;
\end{equation}
\begin{equation}
\label{semisymmid3}
y\backslash x= xy;
\end{equation}
\begin{equation}
\label{semisymmid4}
y/x=xy.
\end{equation}
Let \eqref{semisymmid1} serve as an equational basis of the variety of semisymmetric quasigroups $\mathbf{P}$ relative to the variety of quasigroups $\mathbf{Q}$.  
\begin{remark}
An immediate consequence of the equality of \eqref{semisymmid1}-\eqref{semisymmid4} is that, for semisymmetric quasigroups, the operations $/$ and $\backslash$ coincide with the opposite of multiplication.  
\end{remark}
\subsection{Quasigroup homotopy and semisymmetrization}
Associated with any quasigroup $(Q, \cdot, /, \backslash)$ is its \emph{semisymmetrization}, a semisymmetric quasigroup defined on the set $Q^3$ via the product
\begin{equation}
\label{semisymmetrization}
(x_1, x_2, x_3)(y_1, y_2, y_3)=(y_3/x_2, y_1 \backslash x_3, x_1\cdot y_2).
\end{equation}
These semisymmetrizations are connected to the notion of quasigroup \emph{homotopy}; we declare a triple of functions $(f_1, f_2, f_3):Q\to P$ between quasigroups $Q$ and $P$ to be a homotopy if $xf_1\cdot yf_2=(xy)f_3$, whenever $x, y\in Q.$\footnote{The notion of homotopy gains traction over that of homomorphism in applications of quasigroups to geometry.  Indeed, quasigroup multiplications coordinatize $3$-nets, and in this web-geometric setting, the notion of homomorphism is too strong to describe properly the preservation and translation of properties between nets.}  Such triples constitute the morphism class of the category $\mathbf{Qtp}$ of homotopies between quasigroups.  There is a forgetful functor $\Sigma:\mathbf{P}\to \mathbf{Qtp}$, sending each homomorphism $f: Q\to P$ of semisymmetric quasigroups to the triple $(f, f, f):Q\to P.$  In \cite{HASoQ}, Smith demonstrates that $\Sigma$ is left adjoint to the process of semisymmetrization.  In order to make the construction given by \eqref{semisymmetrization} functorial, we observe that if $(f_1, f_2, f_3):Q\to P$ is a homotopty of quasigroups, then the map $Q^3\to P^3; (x_1, x_2, x_3)\mapsto (x_1f_1, x_2f_2, x_3f_3)$ is a homomorphism of semisymmetric quasigroups (for proof of functoriality, see \cite{HASoQ}, sec 4).  We denote the semisymmetrization functor by $\Delta: \mathbf{Qtp}\to \mathbf{P}.$  

\subsection{Mendelsohn triple systems}
Combinatorial design theory further motivates an examination of semisymmetry in quasigroups.  Indeed, a \emph{Mendelsohn triple system} (MTS) $(P, \mathcal{B})$ consists of a finite point set $P$, together with a collection $\mathcal{B}$ of \emph{triplets} of points with the property that any ordered pair of distinct points belongs to exactly one triplet.  For example, let $(x\phantom{\mid} y\phantom{\mid} z)$ be a prototypical member of $\mathcal{B}$; notice that $(x, y),$ $(y, z),$ and $(z, x)$ are the only ordered pairs appearing in the triple.  Moreover, no distinction should be made between $(x\phantom{\mid} y\phantom{\mid} z),$ $(y\phantom{\mid} z\phantom{\mid} x),$ and $(z\phantom{\mid} x\phantom{\mid} y).$   
\begin{example}
\label{MTSex}
There exist MTS of orders 3 and 4.  Let $P=\{e, a, b\}$.  Then $\mathcal{B}=\{(e\phantom{\mid}a\phantom{\mid}b), (e\phantom{\mid}b\phantom{\mid}a)\}$ makes $(P, \mathcal{B})$ an MTS.\\  
If $P^\prime=\{e, a, b, c\}$, then $\mathcal{B}^\prime=\{(e\phantom{\mid}a\phantom{\mid}b), (e\phantom{\mid}c\phantom{\mid}a), (a\phantom{\mid}c\phantom{\mid}b), (c\phantom{\mid}e\phantom{\mid}b)\}$ produces a $4$-element MTS $(P^\prime, \mathcal{B}^\prime).$   
\end{example}
Mendelsohn triple systems give rise to idempotent, semisymmetric quasigroups on the underlying point set.  Mendelsohn proves this in \cite{MTS}.  Indeed, suppose $(P, \mathcal{B})$ is an MTS.  Define a binary multiplication $\cdot$ as follows: $x\cdot x=x$ for all $x\in P,$ and if $x, y\in P$ are distinct, then $x\cdot y=z$ if and only if $(x\phantom{\mid}y\phantom{\mid}z)\in \mathcal{B}.$  Let $\circ:P^2\to P;(x, y)\mapsto y\cdot x$ be the opposite of this multiplication.  Then $(P, \cdot, \circ, \circ)$ is an idempotent semisymmetric quasigroup.  That is, for distinct $x, y\in P,$ there are distinct blocks $(x\phantom{\mid} y\phantom{\mid} z), (y\phantom{\mid}x\phantom{\mid}z^\prime)\in \mathcal{B}$, and thus the quasigroup identities 
\begin{align*}
\text{(IL)}\phantom{A} y\circ(yx)=yx\cdot y=z^\prime y=x, && \text{(IR)}\phantom{A} (xy)\circ y=y\cdot xy=yz=x,\\ 
\text{(SL)}\phantom{A} y\cdot (y\circ x)=y\cdot xy=yz=x, && \text{(SR)}\phantom{A} (x\circ y)y=yx\cdot y=z^\prime y=x.
\end{align*}
are satisfied.  Idempotence of $\cdot$ and $\circ$ ensure the quasigroup identities are satisfied in the case that $x=y.$  Note too that the semisymmetric identity \eqref{semisymmid1} is verified in (IL).\\
\indent Conversely, if one has an idempotent semisymmetric quasigroup $(P, \cdot, /, \backslash)$ of order at least $3$, we can build an MTS.  Take the ordered pair $(x, y)\in P^2$ of distinct elements.  If $xy=x,$ then $y=(xy)x=x^2=x,$ a contradiction, so $xy\neq x.$  Similarly, we know $xy\neq y.$  Declare $xy=z,$ and $(x\phantom{\mid}y\phantom{\mid}z)$ becomes a cyclic triple of distinct points.  Hence, $(P, \mathcal{B})$ realizes an MTS, with $\mathcal{B}=\{(x\phantom{\mid}y\phantom{\mid} xy)\mid x\neq y\in P\}$.  \\
\indent The above discussion establishes the following theorem.        
\begin{theorem}
There is a one-to-one correspondence between Mendelsohn triple systems and idempotent, semisymmetric quasigroups of order at least $3$.    
\end{theorem}
\begin{example}
The designs of Example \ref{MTSex} appear in Figure 1 as quasigroup multiplication tables.
\end{example}

\begin{figure}
\label{MTStables}
\begin{tabular}{|c||c|c|c|}
\hline
$(P, \cdot)$ & $e$ & $a$ & $b$\\
\hline
\hline
$e$ & $e$ & $b$ & $a$\\
\hline
$a$ & $b$ & $a$ & $e$\\
\hline
$b$ & $a$ & $e$ & $b$\\
\hline
\end{tabular}
\quad
\begin{tabular}{|c||c|c|c|c|}
\hline
$(P^\prime, \cdot)$ & $e$ & $a$ & $b$ & $c$\\
\hline
\hline
$e$ & $e$ & $b$ & $c$ & $a$\\
\hline
$a$ & $c$ & $a$ & $e$ & $b$\\
\hline
$b$ & $a$ & $c$ & $b$ & $e$\\
\hline
$c$ & $b$ & $e$ & $a$ & $c$  \\
\hline
\end{tabular}
\caption{Mendelsohn triple systems of orders 3 and 4}
\end{figure}

\label{MTSs}
By examining idempotent, semisymmetric extensions of partial Latin squares, Mendelsohn was able to prove the following:
\begin{theorem}
A Mendelsohn triple system of order $n$ exists for all $n\neq 2\mod 3$ except for $n=1$ and $n=6$.   
\end{theorem}

For our own purposes, we draw the following:

\begin{corollary}
Let $n\neq 2 \mod 3$ be a positive integer.  If $n\neq 6$, then there is an idempotent, semisymmetric quasigroup of order $n$. 
\end{corollary}

\section{Multiplication groups}

In a quasigroup $Q$, fix an element $q$.  The identities (IL) and (SL) guarantee that the map $L(q):Q\to Q; x\mapsto qx$ is an element of the permutation group $Q!$; similarly, (IR) and (SR) place $R(q)\in Q!,$ where $R(q): x\mapsto xq.$  Let $R(Q)=\{R(q)\mid q\in Q\}$ and $L(Q)=\{L(q)\mid q\in Q\}.$  The subgroup of $Q!$ generated by $R(Q)\cup L(Q)$, denoted by $\text{Mlt(Q)},$ is the $\emph{combinatorial multiplication group}$ of $Q$.  It is worth noting that $R(q)^{-1}:x\mapsto x/q$, while $L(q)^{-1}:x\mapsto q\backslash x.$  \\
\indent If $P$ is a subquasigroup of $Q$, then the \emph{relative multiplication group} of $P$ in $Q$, $\text{Mlt}_Q(P),$ is the subgroup of $Q!$ generated by $\{R(p), L(p)\mid p\in P\}$.

\subsection{Universal multiplication groups}
In building up our definition of a quasigroup module, we are interested in the action of quasigroup multiplications on $Q[X],$ an object which, as the bracketing indicates, is analogous to a ring of polynomials over an indeterminate $X$. \\
\indent Let $\mathbf{V}$ be a variety of quasigroups.  As a category of algebras defined by identities, $\mathbf{V}$ is bicomplete and permits the specification of free objects. For a $\mathbf{V}$-quasigroup, we define $Q[X]$ to be the coproduct of $Q$ with $Q^{\mathbf{V}}_X,$ the free $\mathbf{V}$-quasigroup on the singleton $\{X\}$.  This object comes with insertions $\iota_Q:Q\to Q[X]$ and $\iota_X:Q^{\mathbf{V}}_{X}\to Q[X].$  Furthermore, for any $\mathbf{V}$-quasigroup homomorphism $f:Q\to P$, and element $p\in P$, there is a unique morphism $f_p:Q[X]\to P$ of $\mathbf{V}$ so that $f=\iota_Qf_p$.  This universal property is best illustrated via the combination of coproduct and free object diagrams below: \\     
\begin{displaymath}
\xymatrix{
Q \ar[r]^{\iota_Q} \ar[dr]_f & Q[X] \ar@{-->}[d]_{f_p} & Q^{\mathbf{V}}_{X}\ar[l]_{\iota_X} \ar@{-->}[dl]^{\bar{p}} \\
 & P & \{X\}.\ar@{^{(}->}[u] \ar[l]^p
}
\end{displaymath}

Just as a polynomial ring $R[X]$ contains $R$-linear combinations over an indeterminate, $Q[X]$ consists of quasigroup-operational combinations of the elements of $Q$ over the indeterminate $X$.  
\begin{example}
Consider the quasigroup $P$ of Example \ref{MTSs}.  Now $P$ is semisymmetric; consider $P[X]$ in the category of semisymmetric quasigroups $\mathbf{P}$.  We list some words in $P[X]$:
\begin{align*}
(eX)e=e(Xe)=X;\\
(eX)(eb)=(eX)a;\\
b((XX^2)X)=bX^2.
\end{align*}
In order to see the variety-dependence of this construction, recall that $P$ is also idempotent.  Let $\mathbf{MTS}$ denote the category of semisymmetric, idempotent quasigroups.  If we take $P[X]$ in $\mathbf{MTS},$ then the third word above further reduces: $b((XX^2)X)=bX^2=bX$.
\end{example}

\begin{remark} 
\label{reductions}
We note that words in $Q[X]$ (coproduct taken in $\mathbf{P}$) do have a normal form (see \cite{ENFTR}), and it is based partially on an ordering of quasigroup operations $\{\cdot</<\backslash\}$.  Lesser operations are given preferential treatment.  That is to say, if $u$ and $v$ are fully reduced words in $Q[X]$, then $uv$ will represent $v/u=v\backslash u=uv$.
\end{remark}

As it turns out, the insertion $\iota_Q:Q\to Q[X]$ is a monomorphism; this is clear if $Q$ is empty, and if $Q$ is nonempty, take $f_q:Q[X]\to Q$, for some $q\in Q$, as a retraction of the identity $1_Q:Q\to Q$; indeed, $\iota_Qf_q=1_Q$.  The key to this observation is that it tells us $Q[X]$ contains an isomorphic copy of $Q$.  Therefore, we may define the relative multiplication group $\text{Mlt}_{Q[X]}(Q),$ and this group is the $\emph{universal multiplication}$ group of $Q$ in $\mathbf{V}.$  This construction is functorial (see \cite{IQTR}, Prop 2.8); we denote this functor $U(Q; \mathbf{V}):\mathbf{V}\to \mathbf{Gp}$, but when the specific variety $\mathbf{V}$ and quasigroup $Q$ are clear from context or irrelevant, we'll let $U(Q; \mathbf{V})=\widetilde{G}.$  Moreover, we shall use $\widetilde{R}(q)$ and $\widetilde{L}(q)$ to denote multiplications acting on the set $Q[X]$ as opposed to $Q$.  As a final notational remark, we note that when we want to avoid committing to left or right multiplication, $\widetilde{E}(q)^{\varepsilon}$ will be used for generic length-one words of $\widetilde{G}$, where $\widetilde{E}(q)$ is $\widetilde{R}(q)$ or $\widetilde{L}(q)$, and $\varepsilon=\pm1.$

\begin{example}
\label{gpumult}
Given a group $Q$ in the variety $\mathbf{Gp}$, the structure of the universal multiplication group $\widetilde{G}=U(Q; \mathbf{Gp})$, is captured by the isomorphism $$T:Q\times Q\to \widetilde{G}; (g, h)\mapsto \widetilde{L}(g)^{-1}\widetilde{R}(h).$$
\end{example}

\begin{remark}
It had been previously shown that $U(Q; \mathbf{Q})$ is free on the disjoint union $R(Q)+L(Q)$ of left and right multiplications (see \cite{IQTR}, Th. 2.2).  For semisymmetric quasigroups, the coincidence of left and right divisions with the opposite of multiplication ensure that $U(Q; \mathbf{P})$ is free on $\widetilde{R}(Q).$  This notion is formalized in the proofs that follow.  
\end{remark}

\begin{lemma}
\label{lemma1}
Let $Q$ be a semisymmetric quasigroup, $QG$ the free group on the set $Q$, and $Q[X]$ the $\mathbf{P}$-coproduct of $Q$ with $Q^{\mathbf{P}}_X$.  Consider the unique group homomorphism arising out of $q\mapsto \widetilde{R}(q);$ that is,
\begin{equation}
\label{iso1}
R:QG\to \widetilde{G}; q_1^{\varepsilon_1}\cdots q_R^{\varepsilon_r}\mapsto \widetilde{R}(q_1)^{\varepsilon_1}\cdots \widetilde{R}(q_r)^{\varepsilon_r},
\end{equation}
where each $\varepsilon_i=\pm1.$ If $w=q_1^{\varepsilon_1}\cdots q_r^{\varepsilon_r}$ is a fully reduced word in $QG$, then so is $Xw^R=X(\widetilde{R}(q_1)^{\varepsilon_1}\cdots \widetilde{R}(q_r)^{\varepsilon_r})$ in $Q[X]$.  
\end{lemma}

\begin{proof}
Remark \ref{reductions} implies we need only rule out semisymmetric reductions of the form $(uv)u\to v$ or $u(vu)\to v$.  The cases where the length of $w$ is zero (the identity of $QG$) and one are obvious.  We are left to argue instances in which $r\geq 2$.  By way of contradiction, suppose $u$ and $v$ are nonempty subwords of $Xw^R$ and that so is $(uv)u.$  The bracketing implicit in $Xw^R$ demands $X$ be a subword of $(uv)u$; this is because the closing off of any set of parentheses includes a word that is the result of successive multiplications applied to $X$.  Furthermore, $X$ only appears once in $Xw^R,$ so it must be a subword of $v$, and $u$ is a single element of $Q$; call it $q$.  Let $0\leq s, t\leq r$ such that $\{L_i\}_{i<s}$ and $\{R_i\}_{i<t}$ are (potentially empty) subsequences of $1, \dots, r$ and $v=q_{L_1}\cdots q_{L_s}Xq_{R_1}\cdots q_{R_t}$ (bracketing implicit).  Then $(uv)u=(qv)q=(v\widetilde{R}(q)^{-1})\widetilde{R}(q)$, and $q^{-1}q$ is a subword of $w$, contradicting this free group word's irreducibility.  A similar argument shows that if we assume $u(vu)$ appears in $Xw^R$, then $qq^{-1}$ is a subword of $w$.  Conclude $Xw^R$ is fully reduced in $Q[X].$
\end{proof}

\begin{theorem}
\label{unimult}
Let $Q$ be a quasigroup in the variety $\mathbf{P}$ of semisymmetric quasigroups.  Then the universal multiplication group $\widetilde{G}=U(Q, \mathbf{P})$ is free over the set $\widetilde{R}(Q)=\{\widetilde{R}(q)\mid q\in Q\}$.    
\end{theorem}

\begin{proof}
It suffices to show that \eqref{iso1} bijects.  Begin by noticing that elements of $\widetilde{G}$ are of the form $\widetilde{E}(q_1)^{\varepsilon_1}\cdots \widetilde{E}(q_r)^{\varepsilon_r}$.  Due to \eqref{semisymmid3} and \eqref{semisymmid4}, $\widetilde{L}(q)^{-1}=\widetilde{R}(q)$ and $\widetilde{L}(q)=\widetilde{R}(q)^{-1}$ for all $q\in Q$.  Hence, $\widetilde{E}(q_1)^{\varepsilon_1}\cdots \widetilde{E}(q_r)^{\varepsilon_r}=\widetilde{R}(q_1)^{\varepsilon_1}\cdots \widetilde{R}(q_r)^{\varepsilon_r}$, which is clearly the image of $q_1^{\varepsilon_1}\cdots q_r^{\varepsilon_r}$ under $R$, so the map is surjective.  Next, we illustrate injectivity.  Suppose that $w=q_1^{\varepsilon_1}\cdots q_r^{\varepsilon_r}\in \text{Ker} R$.  By Lemma \ref{lemma1}, $X=Xw^R$ is fully reduced in $Q[X]$, so it must be the case that $r=0$, and $w=1.$   
\end{proof}

\subsection{Universal stabilizers}
We return to the setting of a quasigroup $Q$ in a general variety $\mathbf{V},$ noticing that since we identify $Q$ as a subset of $Q[X],$ we have an action of $U(Q; \mathbf{V})=\widetilde{G}$ on $Q$, and quasigroup axioms reveal this action to be transitive.  For a fixed element $e\in Q,$ we refer to its stabilizer $\widetilde{G}_e$ under this $\widetilde{G}$-action as \emph{the universal stabilizer} of $Q$ in $\mathbf{V}$.  Elements of $\widetilde{G}_e$ have a nice visual interpretation in the Cayley graph afforded by $\widetilde{G}\curvearrowright  Q[X]$.  Consider 
\begin{equation}
\label{Tstab}
\widetilde{T}_e(q)=\widetilde{R}(e\backslash q)\widetilde{L}(q/e)^{-1};
\end{equation}
\begin{equation}
\label{Rstab}
\widetilde{R}_e(q, r)=\widetilde{R}(e\backslash q)\widetilde{R}(r)\widetilde{R}(e\backslash qr)^{-1};
\end{equation}
\begin{equation}
\label{Lstab}
\widetilde{L}_e(q, r)=\widetilde{L}(q/e)\widetilde{L}(r)\widetilde{L}(rq/e)^{-1}.
\end{equation}
The action of \eqref{Tstab} on $e$ amounts to a circuit of two nodes starting at our base vertex $e$ and moving forwards along the edge of right multiplication by $e\backslash q$  to $q$ and returning to $e$ backwards along the edge of left multiplication by $q/e$.  The actions of \eqref{Rstab} and \eqref{Lstab} are three point trips, with the former moving from $e$ to $q$, then along to a possibly distinct point $qr$, and returning to $e$ backwards along the edge directing $e$ to $qr$ via right multiplication by $e\backslash(qr)$.  Then \eqref{Lstab} is a similar trip along arrows given by left multiplications.    
The circuits are pictured below.
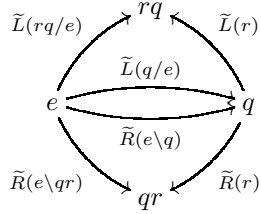
\begin{displaymath}
\xymatrix{
& rq & \\
e \ar@/^/[ur]^{\widetilde{L}(rq/e)} \ar@/^/[rr]^{\widetilde{L}(q/e)} \ar@/_/[dr]_{\widetilde{R}(e\backslash qr)} \ar@/_/[rr]_{\widetilde{R}(e\backslash q)}  & & q \ar@/_/[ul]_{\widetilde{L}(r)} \ar@/^/[dl]^{\widetilde{R}(r)}  \\
& qr& 
}
\end{displaymath}

These elements of the stabilizer are prototypical in the sense that for a quasigroup $Q$ with base point $e$, $U(Q; \mathbf{Q})_e$ is free on $\{\widetilde{T}_e(q), \widetilde{R}_e(q, r), \widetilde{L}_e(q, r)\mid q, r\in Q\}$.  For a proof of this, see Theorem 2.3 of \cite{IQTR}.  Here, Schreier's Theorem is employed, adopting the notation of \cite{Trees}.  In our proof of Theorem \ref{ustab} below, we also follow this notation.  

\begin{theorem}
\label{ustab}
Let $Q$ be a nonempty, semisymmetric quasigroup in $\mathbf{P}$.  Fix $e\in Q$, and let $Q^\#=Q\smallsetminus\{e\}.$  Then the universal stabilizer $\widetilde{G}_e$ is free over the set 
\begin{equation}
\label{ustabbasis}
\left\{\widetilde{R}(e^2), \widetilde{R}(xe)\widetilde{R}(y)\widetilde{R}(xy\cdot e)^{-1}, \widetilde{R}(xe)\widetilde{R}(ex)\mid (x, y)\in Q^\#\times Q, y\neq ex\right\}.  
\end{equation}
If $Q$ is finite, of order $n$, then $\widetilde{G}_e\leq\widetilde{G}$ is a free subgroup of rank $n^2-n+1$. 
\end{theorem}

\begin{proof}
Suppose that  $S=\{\widetilde{R}(q)\mid q\in Q\}$ is a basis for $\widetilde{G}$.  Next, we let $T=\{1\}\cup S\setminus \{\widetilde{R}(e^2)\}$.  As defined, $T$ is a transversal to  $\widetilde{G}_e$ in $\widetilde{G}$.  Indeed, for distinct $x, y\in Q^\#$, $ye(e\cdot xe)=(ye)x\neq e;$ thus, $\widetilde{R}(xe)\widetilde{R}(ye)^{-1}\notin \widetilde{G}_e$, so the cosets represented by $T$ are distinct.  Moreover, since $T$ is in one-to-one correspondence with $Q$, and $\widetilde{G}$ acts transitively on $Q$, the quotient of $\widetilde{G}$ by the stabilizer $\widetilde{G}_e$ of $e$ in $Q$, $\widetilde{G}_e\backslash\widetilde{G}$, is in one-to-one correspondence with $T$; namely, $T$ is a complete set of coset representatives.  Since each element of $T$ is either a single element of $S$ or the identity, $T$ satisfies the partial product condition required by Schreier's Theorem.  Define $$W=\{(t, s)\in T\times S\mid ts\notin T\}=\{(1, \widetilde{R}(e^2)\}\cup\{(\widetilde{R}(xe), \widetilde{R}(y))\mid x\in Q^\#\}.$$  According to Schreier's Theorem, a set $$A=\{h_{t, s}=tsu^{-1}\mid (t, s)\in W\}$$ will be a basis for $\widetilde{G}_e$, where $u\in T$ is chosen so that $\widetilde{G}_ets=\widetilde{G}_eu.$  First, let $h_{1, \widetilde{R}(e^2)}=\widetilde{R}(e^2)$ (here, $\widetilde{R}(e^2)\in\widetilde{G}_e$, so we can choose $u=1$).  Now, fix $x\in Q^\#$, and consider $(\widetilde{R}(xe), \widetilde{R}(y))\in W.$  Suppose $y\neq ex$.  Then $xy\cdot e\neq x(ex)\cdot e=e^2,$ placing $\widetilde{R}(xy\cdot e)\in T.$  Since $\widetilde{R}(xe)\widetilde{R}(y)\widetilde{R}(xy\cdot e)^{-1}\in\widetilde{G}_e,$  $\widetilde{G}_e\widetilde{R}(xe)\widetilde{R}(y)=\widetilde{G}_e\widetilde{R}(xy\cdot e).$  Hence, in the case $y\neq ex$, we may define $h_{\widetilde{R}(xe)\widetilde{R}(y)}=\widetilde{R}(xe)\widetilde{R}(y)\widetilde{R}(xy\cdot e)^{-1}.$  But if $y=ex$ then we may choose $u=1$, for $\widetilde{R}(xe)\widetilde{R}(ex)\in\widetilde{G}_e.$  In other words, $h_{\widetilde{R}(xe)\widetilde{R}(y)}=\widetilde{R}(xe)\widetilde{R}(y)=\widetilde{R}(xe)\widetilde{R}(ex).$  Note $A$ is precisely the set \eqref{ustabbasis}.\\
\indent If $Q$ is of finite order $n$, $$\text{rank}(H)=\left|\widetilde{G}: \widetilde{G}_e\right|(n-1)+1=|T|(n-1)+1=n^2-n+1$$ follows by the Schreier index formula.
\end{proof}
Adapting \eqref{ustabbasis} to fit the notation of \eqref{Tstab}--\eqref{Lstab}, we conclude that the universal stabilizer of a quasigroup in $\mathbf{P}$ is free over 
\begin{equation}
\label{semiunistab}
\left\{\widetilde{R}(e^2), \widetilde{T}_e(x), \widetilde{R}_e(x, y)\mid (x, y)\in Q^{\#}\times Q, y\neq ex\right\}.
\end{equation}

\begin{example}
\label{Pustab}
Consider the MTS $P$ of Example \ref{MTSex}.  The universal multiplication group of $P$ in $\mathbf{P}$ is the rank-$3$ free group $\widetilde{G}=\langle\widetilde{R}(e), \widetilde{R}(a), \widetilde{R}(b)\rangle.$  The universal stabilizer is the free group of rank $7$ presented by 
\begin{equation}
\label{exunistab}
\widetilde{G}_e=\langle \widetilde{R}(e), \widetilde{T}_e(a), \widetilde{T}_e(b), \widetilde{R}_e(a, e), \widetilde{R}_e(a, a), \widetilde{R}_e(b, e), \widetilde{R}_e(b, b)\rangle.   
\end{equation}
We picture the circuits corresponding to $\widetilde{R}(e),$ $\widetilde{T}_e(a),$ and $\widetilde{R}_e(a, e)$ below:
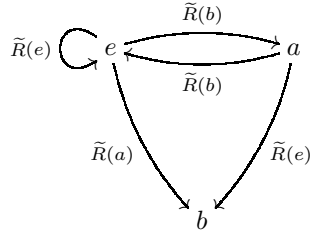
\begin{displaymath}
\xymatrix{
e  \ar@(ul,dl)[]_{\widetilde{R}(e)} \ar@/^/[rr]^{\widetilde{R}(b)} \ar@/_/[ddr]_{\widetilde{R}(a)} & & a\ar@/^/[ll]^{\widetilde{R}(b)}\ar@/^/[ddl]^{\widetilde{R}(e)}\\
 & & \\
& b & 
}
\end{displaymath}
\end{example}


\section{Quasigroup modules}
\subsection{The Fundamental Theorem}
Recall from our Introduction that, for a quasigroup $Q,$ we defined a $Q$-module as an abelian group object in the slice category $\mathbf{Q}/Q$ of morphisms over $Q$.  We have finally gathered sufficient terminology to introduce an equivalent, and much more concrete, formulation of quasigroup module theory:

\begin{theorem}[The Fundamental Theorem of Quasigroup Representations]
\label{FTQR}
Let $Q$ be a quasigroup with element $e$.  Let $\widetilde{G}$ be the universal multiplication group $U(Q, \mathbf{Q})$ of $Q$ in the variety of all quasigroups.  Then $Q$-modules, as abelian group objects in the slice category $\mathbf{Q}/Q$, are equivalent to modules over the universal stabilizer $\widetilde{G}_e$.   
\end{theorem}

The reader will find a proof of the Fundamental Theorem in Section 10.3 of \cite{IQTR}, and although we refrain from reproducing it in full, we will review some of its aspects that we feel help to uncover the rationale of the theorem. \\  \indent Let us outline how, given a quasigroup $Q$ of $\mathbf{Q}$ with element $e$ and a $\widetilde{G}_e$-module $M$, an abelian group object $\pi: E\to Q$ of $\mathbf{Q}/Q$ may be established.  First, we can induce up from the universal stabilizer to give $M$ the structure of a $\widetilde{G}$-module, and the choice of transversal by which we perform the induction makes $E=M\times Q$ a $\widetilde{G}$-set.  Establish a local abelian group structure on $E$ by 
\begin{equation}
\label{localgp}
(m_1, q)-(m_2, q)=(m_1-m_2, q).
\end{equation}
Let $\pi:E\to Q,$ be the projection onto $Q$.  Having to situate $E$ in $\mathbf{Q},$ we note that
\begin{equation}
\label{linear}
\begin{cases}
a\cdot b=a\widetilde{R}(b^\pi)+b\widetilde{L}(a^\pi),\\
a/b= (a-b\widetilde{L}(a^\pi/b^\pi))\widetilde{R}(b^\pi)^{-1},\\
a\backslash b = (b-a\widetilde{R}(a^\pi\backslash b^\pi))\widetilde{L}(a^\pi)^{-1},
\end{cases}
\end{equation}
provides a linearized quasigroup structure on $M\times Q$.  \\
\indent Conversely, if $\pi:E\to Q$ is an abelian group in $\mathbf{Q}/Q$ with addition morphism $+:E\times_Q E\to E$, then $M=\pi^{-1}\{e\}$ forms in abelian group under the restriction of $+$.  Moreover, $0:Q\to E$ injects, making $Q$ a subquasigroup of $E;$ functoriality of $U(-; \mathbf{Q})$ then ensures a projection $\widetilde{G}\to \text{Mlt}_E(Q)$ by which $\widetilde{G}$ acts on $M$, and we can restrict this action to one of $\widetilde{G}_e$ on $M$.    

\subsection{A relativized Fundamental Theorem}
We now discuss a version of Theorem \ref{FTQR} that allows us to realize modules over quasigroups in more structurally specific varieties.  Suppose $Q$ belongs to a variety of quasigroups $\mathbf{V}$.  Such a variety is specified as a subvariety of $\mathbf{Q}$ by a \emph{relative equational basis} of identities.  For example,  
\begin{equation}
\label{semiid}
(yx)y=x
\end{equation}
provides a relative equational basis for $\mathbf{P}$ in $\mathbf{Q}$.  Adding
\begin{equation}
\label{idempo}
x^2=x
\end{equation}
to our equational basis situates us within the variety of Mendelsohn triple systems $\mathbf{MTS}$ introduced in Section 2.2.  The philosophy behind a relativized module theory is to take the ring $\mathbb{Z}\widetilde{G}_e$, that which generates the category of $Q$-modules in $\mathbf{Q}$, and quotient out by an ideal generated by linearized versions of the quasigroup words forming the relative equational basis of $\mathbf{V}$ in $\mathbf{Q}$.  To get a feel for how this works, let's work through this linearization process when attempting to represent the trivial quasigroup $\{e\},$ as an object in $\mathbf{P}$.
\begin{example}
\label{banal}
Let $Q$ denote the trivial quasigroup $\{e\}$.  Theorems \ref{unimult}, \ref{ustab} give $\widetilde{G}=U(Q; \mathbf{P})=\langle \widetilde{R}(e)\rangle\cong\langle \widetilde{R}(e^2)\rangle=\widetilde{G}_e$.  In accordance with \eqref{linear}, define $x\cdot y=x\widetilde{R}(e)+y\widetilde{L}(e)$ whenever $x, y\in\mathbb{Z}\widetilde{G}_e.$  We wish to reflect the identity $(yx)y=x$ in $\mathbb{Z}\widetilde{G}_e,$ so this means that 
\begin{align*}
x&=(y\widetilde{R}(e)+x\widetilde{L}(e))\cdot y\\
&=y\widetilde{R}(e)R(e)+x\widetilde{L}(e)\widetilde{R}(e)+y\widetilde{L}(e)
\end{align*}
for all $x, y\in \mathbb{Z}\widetilde{G}_e,$ including $y=0$.  Therefore, $x=x\widetilde{L}(e)\widetilde{R}(e),$ so as we suspected, $\widetilde{R}(e)=\widetilde{L}(e)^{-1}$.  Setting $x=0$ yields \begin{align*}
0&=y\widetilde{R}(e)^2+y\widetilde{L}(e)\\
&=y\widetilde{R}(e)^2+y\widetilde{R}(e)^{-1}.
\end{align*}  
That is, 
\begin{equation}
\label{linsemiid}
\widetilde{R}(e)^3+1=0 
\end{equation}
in our desired ring of representation.  Conversely, a routine calculation verifies that the linearized multiplication of \eqref{linear} in 
\begin{equation}
\label{banalrepring}
\mathbb{Z}\widetilde{G}_e/(\widetilde{R}(e)^3+1)
\end{equation}
is semisymmetric.  Thus, modules over the one-element semisymmetric quasigroup are equivalent to modules over the quotient ring \ref{banalrepring}, which may be more familiar to the reader as a quotient ring $\mathbb{Z}[X, X^{-1}]/(X^3+1)$ of integral Laurent polynomials .  
\end{example}
Hinted at in \eqref{linear} and Example \ref{banal} is a derivation $\frac{\partial}{\partial x}:Q^n\to \mathbb{Z}\widetilde{G}$ on quasigroup words.  Here, we are employing an infix notation.  A quasigroup word is given by $x_1\dots x_nw,$ where $x_1, \dots, x_n$ denote quasigroup elements called the \emph{arguments} of $w$.  For example, we might denote the word $u=(yx)y$ by $x_1x_2u,$ implying that it only takes two elements as arguments.  Section 10.4 of \cite{IQTR} provides a nice exposition of differentiation in quasigroups.  For our purposes, we need only note that
\begin{equation}
\label{prodrule}
\frac{\partial u\cdot v}{\partial x_j}=\frac{\partial u}{\partial x_j}\tilde R(v)+\frac{\partial v}{\partial x_j}\tilde L(u), 
\end{equation}
and
\begin{equation}
\label{deltarule}
\frac{\partial x_i}{\partial x_j}=\delta_{ij} 
\end{equation}
for quasigroup words $w=u\cdot v$, $x_i, x_j\in Q$.  As a final matter of bookkeeping, for quasigroup elements $x, y\in Q,$ define $\rho(x, y)=R(x\backslash x)^{-1}R(x\backslash y).$  The relativized theorem, appearing below, is proved in Section 10.5 of \cite{IQTR}.

\begin{theorem}[Fundamental Theorem for Representations in Varieties]
\label{RFTQR}
Let $Q$ be a nonempty quasigroup in a variety $\mathbf{V}$.  Suppose $B$ is a relative equational basis for $\mathbf{V}$.  Then the category $\mathbb{Z}\otimes \mathbf{V}/Q$ of $Q$-modules in $\mathbf{V}/Q$ is equivalent to the category of modules over the ring $\mathbb{Z}\mathbf{V}Q,$ defined as the quotient of $\mathbb{Z}\widetilde{G}_e$ by the ideal $J\mathbb{Z}\widetilde{G}_e$ generated by
\begin{equation}
\label{J}
\rho(e, q_h)\left(\frac{\partial u}{\partial x_h}(q_1, \dots, q_n)-\frac{\partial v}{\partial x_h}(q_1, \dots, q_n)\right)\rho(e, q_1\cdots q_nu)^{-1}
\end{equation}
for all $u=v$ in $B$, and all $q_h\in Q$.
\end{theorem}

\indent Given a quasigroup word $x_1\dots x_nw,$ an argument $x_h$ is said to occur \emph{uniquely above the line} in $w$ if it appears once in $w$, it appears to the left of any right divisions $/$, and to the right of any left divisions $\backslash$.  By Proposition 10.4 of \cite{IQTR}, any argument occurring uniquely above the line in $u$ and $v$, for $u=v$ in a relative equational basis $B$, does not contribute to the ideal $J\mathbb{Z}\widetilde{G}_e.$  An immediate consequence of this principle is the revelation that group module theory fits within our general theory of quasigroup modules.

\begin{example}
Let $Q$ be a group in $\mathbf{Gp}$.  Recall Example \ref{gpumult}, in which we noted $\widetilde{G}\cong Q\times Q.$  Choosing the identity element $e\in Q$ as our base point, the universal stabilizer is seen to be the diagonal $\widetilde{G}_e=\{(q, q)\mid q\in Q\}\cong Q$.    Now $(x_1x_2)x_3=x_1(x_2x_3)$ is a relative equational basis for $\mathbf{Gp}$.  Each argument of this basis occurs uniquely above the line, so $\mathbb{Z}\mathbf{Gp}Q\cong \mathbb{Z}Q/(0)\cong\mathbb{Z}Q,$ recovering the traditional notion of group module.
\end{example}

\subsection{Representations in the variety $\mathbf{P}$}
In order to determine the ring $\mathbb{Z}\mathbf{P}Q$, whose modules furnish representations of a semisymmetric quasigroup $Q$, we differentiate $(yx)y=x$ as a relative equational basis for $\mathbf{P}$.  Notice that $x$ occurs uniquely above the line; therefore, differentiation with respect to $x$ makes no contribution to $J\mathbb{Z}\widetilde{G}_e$.  Furthermore, 
\begin{align*}
\frac{\partial(yx\cdot y)}{\partial y}&=\frac{\partial(yx)}{\partial y}\widetilde{R}(y)+\frac{\partial y}{\partial y}\widetilde{L}(yx)\\
&=\left(\frac{\partial y}{\partial y}\widetilde{R}(x)+\frac{\partial x}{\partial y}\widetilde{L}(y)\right)\widetilde{R}(y) + \widetilde{R}(yx)^{-1}\\
&=\widetilde{R}(x)\widetilde{R}(y)+ \widetilde{R}(yx)^{-1}.
\end{align*}
Note $\partial x/\partial y=0.$  For $x\in Q$, $\rho(e, x)=R(e\backslash e)^{-1}R(e\backslash x)=\widetilde{R}(e^2)^{-1}\widetilde{R}(xe),$ and, hence, $J\mathbb{Z}\widetilde{G}_e$ is generated by 
\begin{align*}
\rho(e, y)&\left(\widetilde{R}(x)\widetilde{R}(y)+ \widetilde{R}(yx)^{-1}\right)\rho(e, x)^{-1}\\
&=\widetilde{R}(e^2)^{-1}\widetilde{R}(ye)\left(\widetilde{R}(x)\widetilde{R}(y)+ \widetilde{R}(yx)^{-1}\right)\widetilde{R}(xe)^{-1}\widetilde{R}(e^2),
\end{align*}
for $x, y\in Q,$ which is equivalent to the ideal of $\mathbb{Z}\widetilde{G}_e$ generated by $$\widetilde{R}(ye)\left(\widetilde{R}(x)\widetilde{R}(y)+ \widetilde{R}(yx)^{-1}\right)\widetilde{R}(xe)^{-1}$$ whence, our main result:

\begin{theorem}
\label{mainresult}
Let $Q$ be a nonempty, semisymmetric quasigroup.  Suppose $\widetilde{G}=U(Q; \mathbf{P})$, and $e\in Q$, so that $\widetilde{G}_e$ is the universal stabilizer of $Q$ in $\mathbf{P}$.  Let $J\mathbb{Z}\widetilde{G}_e$ be the two-sided ideal of the integral group ring $\mathbb{Z}\widetilde{G}_e$ generated by $$\{\widetilde{R}(ye)(\widetilde{R}(x)\widetilde{R}(y)+\widetilde{R}(yx)^{-1})\widetilde{R}(xe)^{-1}\mid x, y\in Q\}.$$   Then $Q$-modules, as abelian group objects in the slice category $\mathbf{P}/Q$, are equivalent to modules over the quotient ring $\mathbb{Z}\widetilde{G}_e/J\mathbb{Z}\widetilde{G}_e$.
\end{theorem}

\subsection{Extensions of semisymmetric quasigroups}
Theorem \ref{mainresult} gives us a recipe for extending semisymmetric quasigroups.  We conclude with some examples of such extensions.

\begin{example}[The trivial quasigroup]
\label{trivialex}
This example is introduced in \cite{ImKoSm}.  We reproduce it here because of the nice connection it makes between representation theory and semisymmetrization.  Consider the trivial quasigroup $Q=\{e\}$ as semisymmetric.  Demonstrated in \ref{banal} is the fact that the ring of representation for $Q$ in $\mathbf{P}$ becomes $\mathbb{Z}\widetilde{G}_e/(\widetilde{R}(e)^3+1)\cong\mathbb{Z}[X, X^{-1}]/(X^3+1)$.    \\
\indent Let $(A, +, 0)$ be an abelian group; a quasigroup structure on $A$ is given by $(A, +, -, \sim),$ where $x\sim y=y-x$ is the opposite of subtraction.  The semisymmetrization $A^{\Delta}$ is defined on $A^3$ with product $$[x_1, x_2, x_3]\cdot[y_1, y_2, y_3]=[y_3-x_2, x_3-y_1, x_1+y_2].$$  This operation has a matrix interpretation: 
\begin{equation}
\label{Pmateq}
[x_1, x_2, x_3]\cdot [y_1, y_2, y_3]=[x_1, x_2, x_3]E+[y_1, y_2, y_3]E^{-1},
\end{equation}
where  
\begin{equation}
\label{Pmat}
E=\left[
\begin{array}{ccc}
0 & 0 & 1\\
-1 & 0 & 0\\
0 & 1 & 0
\end{array}
\right].
\end{equation}
Notice that $E\in \text{Aut}(A^3)$, and that $E^3+I_3=0$, making $A^3$ a $\mathbb{Z}\widetilde{G}_e/(\widetilde{R}(e)^3+1)$-module, and $A^\Delta$ the split extension on $Q\times A^3=\{e\}\times A^3$, with \eqref{Pmateq} tracking the linearized product of \eqref{linear}.
\end{example}

\begin{example}
Consider $P=\{e, a, b\}$ of Examples \ref{MTSex} and \ref{Pustab}.  Let $\mathbb{F}_3$ be the field of order 3.  We demonstrate that the abelian group structure on the direct product $\mathbb{F}^3_3$ furnishes a module over $P$.  Applying Theorem \ref{mainresult}, we obtain 
\begin{align*}
J\mathbb{Z}\widetilde{G}_e=(\widetilde{R}(e)^3+1&, \widetilde{R}(a)\widetilde{R}(b)^3\widetilde{R}(a)^{-1}+1, \widetilde{R}(b)\widetilde{R}(a)^3\widetilde{R}(b)^{-1}+1,\\
&\widetilde{R}(b)\widetilde{R}(e)\widetilde{R}(a)+1, \widetilde{R}(a)\widetilde{R}(e)\widetilde{R}(b)+1).
\end{align*}
 The assignments
\begin{equation}
\label{E}
\widetilde{R}(e)\mapsto \left[\begin{array}{ccc}
0 & 0 & 1\\
-1 & 0 & 0\\
0 & 1 & 0
\end{array}\right]=:E;
\end{equation}
\begin{equation}
\label{A}
\widetilde{R}(a)\mapsto \left[\begin{array}{ccc}
0 & 1 & -1\\
1 & 0 & -1\\
-1 & -1 & 0
\end{array}\right]=:A;
\end{equation}
\begin{equation}
\label{B}
\widetilde{R}(b)\mapsto \left[\begin{array}{ccc}
-1 & 1 & 1\\
-1 & -1 & -1\\
-1 & 1 & -1
\end{array}\right]=:B,
\end{equation}
specify a map $\{\widetilde{R}(e), \widetilde{R}(a), \widetilde{R}(b)\}\to \text{Aut}(\mathbb{F}^3_3),$ which extends to a group homomorphism $\rho:\widetilde{G}\to \text{Aut}(\mathbb{F}^3_3),$ making $\mathbb{F}^3_3$ a $\widetilde{G}$-module.  Take $\rho_e:\widetilde{G}_e\to \text{Aut}(\mathbb{F}^3_3)$ to be the restriction of the representation $\rho$, yielding a $\mathbb{Z}\widetilde{G}_e$-module.  In order to guarantee that $\rho_e$ may also be interpreted as a $\mathbb{Z}\widetilde{G}_e/J\mathbb{Z}\widetilde{G}_e$-module, we have to ensure that $J\mathbb{Z}\widetilde{G}_e$ annihilates $\mathbb{F}^3_3$.  Indeed, elementary calculations verify that $E$, $A$, and $B$ are order-$6$ elements of $\mathsf{GL}_3(\mathbb{F}_3),$ and so the first three generators of $J\mathbb{Z}\widetilde{G}_e$ annihilate our proposed module; moreover, $BEA=AEB=-I_3$, satisfying the latter two generators of the ideal.  We conclude, then, that $(Q, \cdot, \circ, \circ)$ is a semisymmetric quasigroup, where $Q=\mathbb{F}^3_3\times P,$ and we define the multiplication $$([x_1, x_2, x_3], p)\cdot ([y_1, y_2, y_3], q)=([x_1, x_2, x_3]\widetilde{R}(q)+[y_1, y_2, y_3]\widetilde{R}(p)^{-1}, pq),$$ and $\circ$ to be its opposite.           
\end{example}

\section*{Acknowledgements}
The author is very grateful to his advisor, Jonathan D. H. Smith, for introducing him to the topic and for guiding him through the preparation of both this paper and the talk that inspired it.

\bibliographystyle{amsplain}

\begin{thebibliography}{7}
\bibitem{Beck} Beck, J. M., \emph{Triples, Algebras, and Cohomology}, Ph.D. thesis, Columbia University 1967, \emph{Reprints in Theory and Applications of Categories}, 2, 1--59, 2003.
\bibitem{ImKoSm} Im, Bokhee, Ko, Hayi-Joo, and Smith, J. D. H., Semisymmetrizations of abelian group isotopes, \emph{Taiw. J. of Math.,} \textbf{11} 5 (2007), 1529--1534. 
\bibitem{MTS} Mendelsohn, N. S., A natural generalization of Steiner triple systems, in \emph{Computers in Number Theory}, Academic Press, New York, 1971, 323--338.
\bibitem{Trees} Serre, J. P., \emph{Trees}, Springer, Berlin, 1980.
\bibitem{ENFTR} Smith, J. D. H., Evans' normal form theorem revisited, \emph{Int. J. of Algebra and Comp.}, \textbf{17}, 8 (2007), 1577--1592.
\bibitem{HASoQ} Smith, J. D. H., Homotopy and semisymmetry of quasigroups, \emph{Algebra Univ.}, $\mathbf{38}$ (1997), 175--184. 
\bibitem{IQTR}
Smith, J. D. H., \emph{An Introduction to Quasigroups and their Representations,} CRC Press, Boca Raton, 2007. 
\bibitem{infgrp}
Smith, J. D. H., \emph{Representation Theory of Infinite Groups and Finite Quasigroups}, Universit\'{e} de Montr\'{e}al, Montreal, 1986. 
\end{thebibliography}

\end{document}